\documentclass[10pt]{article}%
\usepackage{amsmath, amsthm}
\usepackage{graphicx}
\usepackage{amsfonts}
\usepackage{amssymb}
\usepackage{times}

\newtheorem{theorem}{Theorem}

\newtheorem{corollary}[theorem]{Corollary}

\newtheorem{definition}[theorem]{Definition}

\newtheorem{lemma}[theorem]{Lemma}

\begin{document}

\begin{center}
\vspace{0.4in}

{\LARGE \textbf{Strange Attractors for Asymptotically Zero Maps}}

\vspace{0.3in}

\smallskip

\textsf{Yogesh Joshi}

\textsf{Department of Mathematics and Computer Science}

\textsf{Kingsborough Community College}

\textsf{Brooklyn, NY 11235-2398}

\textsf{yogesh.joshi@kbcc.cuny.edu}

$\ast\;\ast\;\ast$

\smallskip

\textsf{Denis Blackmore}

\textsf{Department of Mathematical Sciences and}

\textsf{Center for Applied Mathematics and Statistics}

\textsf{New Jersey Institute of Technology}

\textsf{Newark, NJ 07102-1982}

\textsf{deblac@m.njit.edu}
\end{center}

\vspace{0.2in}

\noindent\textbf{ABSTRACT:} A discrete dynamical system in Euclidean $m$-space
generated by the iterates of an asymptotically zero map $f$ , satisfying
$\left\vert f(x)\right\vert \rightarrow0$ as $\left\vert x\right\vert
\rightarrow\infty$, must have a compact global attracting set $A $. The
question of what additional hypotheses are sufficient to guarantee that $A$
has a minimal (invariant) subset $\mathfrak{A}$ that is a chaotic strange
attractor is answered in detail for a few types of asymptotically zero maps.
These special cases happen to have many applications (especially as
mathematical models for a variety of processes in ecological and population
dynamics), some of which are presented as examples and analyzed in
considerable detail.

\bigskip


\medskip


\section{Introduction}

The identification and characterization of (chaotic) strange attractors in
discrete dynamical systems is important for both theory and applications
(see,e.g. \cite{Arr,BBK,Caz,CA,Gh,HKLN,JB2,Mar,Rou,Ru,SR1,SR2,Thun,UW,Wig} ,
but this is often very difficult to accomplish with the requisite mathematical
rigor. Compelling evidence of the difficulty in proving the existence of
strange attractors, even for relatively simple nonlinear maps, is provided by
the pioneering work of Misiurewicz \cite{Mis}\ on the Lozi map and that of
Benedicks \& Carleson \cite{BC} on the H\'{e}non map, which in both cases -
and especially the latter - required exceptionally lengthy, delicate and
inspired analysis. In recent years, the basic ideas behind the proofs of these
two landmark results in strange attractor theory have recently been extended
and generalized in terms of a theory of \emph{rank one maps} in an
extraordinary series of papers by Wang \& Young \cite{WY1,WY2,WY3}, the
content of which gives striking confirmation of the exceptional complexity
underlying characterizations of strange attractors for broad classes of
discrete dynamical systems. Not only are the foundational results of rank one
theory hard to prove, they also tend to be rather difficult to apply, as for
example in Ott \& Stenslund \cite{OS}, which is closely related to results of
Zaslavsky \cite{Zas} and Wang \& Young \cite{WY2a}. In light of this rather
daunting rigorous strange attractor landscape, it is clear that there is a
need for theory that is simpler to develop and apply for special classes of
discrete dynamical systems of significant theoretical and applied interest. We
make a start in this direction for some types among the class of discrete
dynamical systems generated by what we call \emph{asymptotically zero maps},
defined as follows: A map
\[
f:\mathbb{R}^{m}\rightarrow\mathbb{R}^{m}%
\]
is \emph{asymptotically zero} (\emph{AZ}) if $\left\vert f(x)\right\vert
\rightarrow0$ as $\left\vert x\right\vert \rightarrow\infty$. We note that
this includes \emph{exponentially decaying maps} (\emph{cf}. \cite{CB})
satisfying the property that $\left\vert f(x)\right\vert \leq Me^{-\left\vert
x\right\vert }$ for some $M>0$ and all $x\in\mathbb{R}^{m}$, and it was our
earlier work on these kinds of maps \cite{JB1,JB2} that inspired our present
efforts - in fact, we were even able to formulate versions of the theorems
that we prove here.

Our organization of the rest of this paper is as follows: In Section 2 we
provide precise definitions of the concepts employed in the sequel and also
prove some elementary dynamical properties of AZ maps. Then, in Section 3, we
prove the existence of what we call \emph{radial strange attractors} for
smooth AZ maps possessing certain additional special properties. This is
followed in Section 4 by a proof of the existence of another special kind of
strange attractor, which we refer to as being of \emph{multihorseshoe} (or
\emph{trellis}). In Section 5 we apply out theorems to several examples
arising from various mathematical models and illustrate the strange attractors
via simulation. We conclude our investigation in Chapter 6 by summarizing our
findings, attempting to assess their impact and identifying some open problems
they suggest. In addition, we also indicate some of our plans for related
future research on strange attractors.

\section{Preliminaries}

Since there does not seem to universally accepted definitions for some of the
key concepts that concern us, we shall describe exactly the ones we are using.
For more standard definitions, we refer the reader to
\cite{Arr,Dev,Gh,HKLN,Rob,Wig}. First, we shall deal with chaos, which has
several more or less equivalent characterizations.

\begin{definition}
\label{defCh} definition. Let $f:X\rightarrow X$ be a continuous self map of
the metric space $X$. The map $f$, or the (semi) discrete dynamical system
comprised of the iterates $f^{n}$ as $n$ ranges over the natural numbers
$\mathbb{N}$, is \textbf{chaotic} if (1) $f$ is topologically transitive (or
mixing), meaning that for every pair $U$, $V$ of open sets of $X$ there exists
a $k\in\mathbb{N}$ such that $f^{k}(U)\cap V\neq\varnothing$ and (2) the set
of periodic points of $f$, denoted as $Per(f)$, is dense in $X$.
\end{definition}

\noindent This is essentially the prescription of Devaney \cite{Dev} sans the
property of sensitive dependence of iterates - common to most definitions of
chaos. However, as shown by Banks \emph{et al}. \cite{BBCDS}, sensitive
dependence is implied by (1) and (2), which renders it redundant for our definition.

The definitions of an attractor and a strange attractor are also subject to
various interpretations, so let us be specific.

\begin{definition}
\label{defAS} Let $f:X\rightarrow X$ be a continuous self map of the metric
space $X$. A subset $A$ of $X$ is an \textbf{attracting set} of the map $f$
(or dynamical system generated by $f$) if it satisfies the following properties:

\begin{itemize}
\item[]
\begin{itemize}
\item[(AS1)] It is a nonempty, closed and (positively) invariant, i.e.
$f(A)\subset A$.

\item[(AS2)] There exists an open set $U$ containing $A$ such for every $x\in
U$, $d\left(  f^{n}(x),A\right)  \rightarrow0$ as $n\rightarrow\infty$.
\end{itemize}
\end{itemize}
\end{definition}

\noindent Attracting sets with more complex structure and dynamics may be
defined as follows:

\begin{definition}
\label{defCAS} Let $f:X\rightarrow X$ be a continuous self map of the metric
space $X$. A subset $A$ of $X$ is a \textbf{\ chaotic attracting set} of the
map $f$ $($or dynamical system generated by $f)$ if it satisfies the following properties:

\begin{itemize}
\item[]
\begin{itemize}
\item[(CAS1)] It is an attracting set for $f$.

\item[(CAS2)] There is a nonempty closed invariant subset $A_{\ast}$ of $A$
such that the restriction $f_{|A_{\ast}}:A_{\ast}\rightarrow A_{\ast}$ is chaotic.
\end{itemize}
\end{itemize}
\end{definition}

\begin{definition}
\label{defSCAS} Let $f:X\rightarrow X$ be a continuous self map of a subset of
$\mathbb{R}^{m}$that is $C^{1}$ except possibly on finitely many $C^{1}$
submanifolds of $\mathbb{R}^{m}$ of dimensions less than $m$ $($a property we
denote for convenience as $C^{1^{\ast}})$. A subset $A$ of $X$ is a
\textbf{semichaotic attracting set} of the map $f$ $($or dynamical system
generated by $f)$ if it satisfies the following properties:

\begin{itemize}
\item[]
\begin{itemize}
\item[(SCAS1)] It is a compact attracting set for $f$.

\item[(SCAS2)] The restriction $f_{|A}:A\rightarrow A$ is continuous and
$C^{1^{\ast}}$ and \textbf{semichaotic} in the sense that there is a nonempty,
invariant subset $A_{\ast}$ of $A$ on which it is sensitively dependent on
initial conditions, i.e.
\[
n^{-1}\log\left\Vert (f^{n})^{\prime}(x)\right\Vert =n^{-1}{\displaystyle\sum
\nolimits_{k=1}^{n}}\log\left\Vert f^{\prime}\left(  f^{k-1}(x)\right)
\right\Vert \geq\ell>0\;\forall(x,n)\in A_{\ast}\times\mathbb{N},
\]
where $\left\Vert \cdot\right\Vert $denotes the standard norm of a linear map.\
\end{itemize}
\end{itemize}
\end{definition}

\noindent We note that it follows from the multiplicative ergodic theorem of
Oseledec (see, \emph{e.g.}\cite{Rob,Wig}) that the expression in (SCAS2) has a
limit almost everywhere on $\mathbb{R}^{m}$, which represents the maximum
Lyapunov exponent of the orbit initiating at $x$.

For an attractor, we add a bit more including a minimality requirement.

\begin{definition}
\label{defA} Let $f:X\rightarrow X$ be a continuous self map of the metric
space $X$. A subset $\mathcal{A}$ of $X$ is an \textbf{attractor} of the map
$f$ $($or dynamical system generated by $f)$ if it satisfies the following properties:

\begin{itemize}
\item[]
\begin{itemize}
\item[(A1)] It is a nonempty, closed and (positively) invariant, i.e.
$f(\mathcal{A})\subset\mathcal{A}$.

\item[(A2)] There exists an open set $U$ containing $\mathcal{A}$ such for
every $x\in U$, $d\left(  f^{n}(x),\mathcal{A}\right)  \rightarrow0$ as
$n\rightarrow\infty$.

\item[(A3)] It is minimal with respect to (i) and (ii), i.e. $\mathcal{A}$ has
no proper subset with these properties.
\end{itemize}
\end{itemize}
\end{definition}

\noindent A strange attractor is an attractor with additional properties
including, for our purposes, chaotic dynamics.

\begin{definition}
\label{defCSA} Let $f:X\rightarrow X$ be a continuous self map of the metric
space $X$. A subset $\mathfrak{A}$ of $X$ is a \textbf{chaotic strange
attractor} of the map $f$ $($or dynamical system generated by $f)$ if it
satisfies the following properties:

\begin{itemize}
\item[]
\begin{itemize}
\item[(CSA1)] It is an attractor.

\item[(CSA2)] The set is fractal, i.e. it has a noninteger fractal (Hausdorff) dimension.

\item[(CSA3)] The restriction $f|_{\mathfrak{A}}$ is chaotic.
\end{itemize}
\end{itemize}
\end{definition}

\noindent The last property (CSA3) is not always included in definitions of
strange attractors, and there are examples (such as in Grebogi \emph{et al}.
\cite{GOPY}) in which (CSA1) and (CSA2) are satisfied, but $f|_{\mathfrak{A}}
$ is regular (nonchaotic).

For our purposes, it is convenient to define a weaker form of chaotic strange attractor.

\begin{definition}
\label{defSCSA} Let $f:X\rightarrow X$ be a continuous self map of a subset of
$\mathbb{R}^{m}$. A subset $\mathfrak{A}$ of $X$ is a \textbf{semichaotic
strange attractor} of the map $f$ (or dynamical system generated by $f$) if it
satisfies the following properties:

\begin{itemize}
\item[]
\begin{itemize}
\item[(SCSA1)] It is an attractor.

\item[(SCSA2)] The set is fractal, i.e. it has a noninteger fractal
$($Hausdorff$)$ dimension.

\item[(SCSA3)] The restriction $f|_{\mathfrak{A}}$ is $C^{1^{\ast}}$ and semichaotic.
\end{itemize}
\end{itemize}
\end{definition}

\subsection{Basic dynamical properties of AZ maps}

Here we consider a couple of useful dynamical properties of continuous AZ maps
that are easy to prove.

\begin{lemma}
\label{lem2.1.1} If $f:\mathbb{R}^{m}\rightarrow\mathbb{R}^{m}$ is a
continuous AZ map, $f$ and all of its iterates $f^{n}$, $n>1$, have fixed
points in the compact ball $B_{R}(0):=\{x\in\mathbb{R}^{m}:\left\vert
x\right\vert \leq R\}$ for $R>0$ sufficiently large.
\end{lemma}

\begin{proof}
As $\left\vert f\right\vert $ is continuous and $f$ is an AZ map, $%
\left\vert f\right\vert $ must achieve a maximum, say $M>0$, on $\mathbb{R}%
^{m}$. Accordingly $f^{n}(B_{R}(0))\subset B_{M}(0)\subset B_{R}(0)$ for
every $R\geq M$ and $n\in \mathbb{N}$, which completes the proof in virtue
of Brouwer's fixed point theorem (see, \emph{e.g}. \cite{Hat}).
\end{proof}

\begin{lemma}
\label{lem2.1.2} Let $f$ and $M$ be as in Lemma \ref{lem2.1.1} and its proof.
Then $f$ has a compact globally attracting set defined as%
\begin{equation}
A:={\displaystyle\bigcap\nolimits_{n=1}^{\infty}}f^{n}\left(  B_{M}(0)\right)
\subset B_{M}(0),\label{e1}%
\end{equation}
and this set must contain all of the fixed points of $f$ and its iterates.
\end{lemma}

\begin{proof}
The set $A$ is the intersection of compact sets in $\mathbb{R}^{m}$, so it
must be compact. Furthermore, it follows from Lemma \ref{lem2.1.1} that it
must be contained in $B_{M}(0)$ and contain all fixed points of $f$ and its
iterates. Thus, the proof is complete.
\end{proof}

\noindent As an immediate corollary of the above lemmas, we obtain the
following relation among the standard recurrence related sets (see e.g.
\cite{Shub}); namely, the periodic points $Per(f)$, the \emph{positive limit
set} $L_{+}(f)$ comprised of the closure of the union of all\emph{\ }$\omega
$\emph{-limit sets}, the \emph{nonwandering set} $\Omega(f)$ and the
\emph{chain-recurrent set} $R(f)$:

\begin{corollary}
\label{cor2.1.1} If $f$, $M$ and $A$ are as in Lemmas \ref{lem2.1.1} and
\ref{lem2.1.2}, then%

\begin{equation}
Per(f)\subset L_{+}(f)=\Omega(f)=R(f)=A.\label{e2}%
\end{equation}

\end{corollary}

It follows from Lemma \ref{lem2.1.2} that $A$ contains at least one minimal
set (defined as a nonempty, closed invariant set with no proper subsets having
the same properties). In the remaining sections we identify some cases when
one of the minimal subsets is also a strange attractor. However, we first
pause to prove a simple result about attractors that are at the other end of
the complexity spectrum on which we are focusing, but have something of the
radial essence of the first type of strange attractor we \ shall analyze in
the next section.

\begin{lemma}
\label{lem2.1.3} Suppose $f$, $M$ and $A$ are as in Lemmas \ref{lem2.1.1} and
\ref{lem2.1.2}, and satisfy the following additional properties:
$(\mathrm{i})$ $f(0)=0$, and $(\mathrm{ii})$ $\left\vert f(x)\right\vert
<\left\vert x\right\vert $ for all $0<\left\vert x\right\vert \leq R_{M}$,
where $R_{M}:=\max\{\left\vert x\right\vert :\left\vert f(x)\right\vert =M\}$.
Then the origin $\{0\}$ is a global attractor for $f$.
\end{lemma}

\begin{proof}
Let $x_{0}$ be any initial point in $\mathbb{R}^{m}$. If $x_{0}$ or any of its
$f$-iterates $x_{k}:=f^{k}(x_{0})$ is zero, we are done; so we assume that
$x_{k}\neq0$ for all $k\in\mathbb{N}$. It follows from the definitions and
hypotheses that $0<\left\vert x_{k}\right\vert \leq R_{M}$ for all $k\geq1$
and that $\left\{  \left\vert x_{k}\right\vert :k\in\mathbb{N}\right\}  $ is a
strictly decreasing sequence of positive numbers. We need only show that
$\left\vert x_{k}\right\vert \rightarrow0$ as $k\rightarrow\infty$. If not,
the monotone convergence property of the reals implies that the sequence
converges to some number $a$ satisfying $0<a<R_{M}$. But this is impossible
owing to the continuity of $\left\vert x\right\vert -\left\vert
f(x)\right\vert $ and the compactness of $\{x\in\mathbb{R}^{m}:$
$a/2\leq\left\vert x\right\vert \leq(a+R_{M})/2\}$. Hence, we conclude that
$\left\vert x_{k}\right\vert \rightarrow0$ as $k\rightarrow\infty$, which
completes the proof.
\end{proof}

\section{Radial Strange Attractors}

Our first result in this section is for a special class AZ maps.

\begin{definition}
\label{EZ1} An AZ map $f:\mathbb{R}^{m}\rightarrow\mathbb{R}^{m}$ such that
there exists a positive $R$ for which $f(x)=0$ whenever $\left\vert
x\right\vert \geq R>0$ is said to be \textbf{eventually zero (EZ)}.
\end{definition}

\noindent If say the map was devised to for applications in population
dynamics, an EZ model might be used in cases where all the species rapidly
become extinct if the sum of all their members becomes too large. More
specifically, the result that follows would apply to ecological dynamics
models for what are called \emph{pioneer species} that satisfy such an
extinction property, while the subsequent theorem should be applicable to
ecological dynamics involving \emph{climax species }(see
\cite{JEAA,Frank,JB1,JB2,JF,Seln,Sum}).

\subsection{Attractors for EZ maps expanding at the origin}

If the origin is a source for an EZ map as is often the case for models of
pioneer species, we have the following result.

\begin{theorem}
\label{thm3.1} Let $f:\mathbb{R}^{m}\rightarrow\mathbb{R}^{m}$ be a continuous
EZ map, with $M$ and $R_{M}$ as in Lemma \ref{lem2.1.3}, satisfying the
following additional properties:

\begin{itemize}
\item[(i)] $f^{-1}\left(  \{0\}\right)  =\{0\}\cup Z$, where $Z=\{x\in
\mathbb{R}^{m}:\left\vert x\right\vert \geq\zeta(x/\left\vert x\right\vert
)>0$ $\}$, $\zeta:\mathbb{S}^{m-1}\rightarrow\mathbb{R}$ is a $C^{1}$ function
satisfying $R_{M}<\zeta(u)<M$ for all $u\in\mathbb{S}^{m-1}$ and
$\mathbb{S}^{m-1}:=\{u\in\mathbb{R}^{m}:\left\vert u\right\vert =1\}$ is the
unit $(m-1) $-sphere.

\item[(ii)] The set $S_{\ast}:=f^{-1}\left(  Z\right)  $ is an $(m-1)$%
-spherical shell of the form%
\[
S_{\ast}=\{x\in\mathbb{R}^{m}:0<\alpha(x/\left\vert x\right\vert
)\leq\left\vert x\right\vert \leq\beta(x/\left\vert x\right\vert )\},
\]
where $\alpha,\beta:\mathbb{S}^{m-1}\rightarrow\mathbb{R}$ are positive
$C^{1}$ functions such that $0<\beta(u)-\alpha(u)<\zeta(u)$ for all
$u\in\mathbb{S}^{m-1}$.

\item[(iii)] $f$ is $C^{1}$ in $D:=\{0\}\cup\{x\in\mathbb{R}^{m}:0<\left\vert
x\right\vert <\zeta\left(  x/\left\vert x\right\vert \right)  \}$ and the
derivative $f^{\prime}(x)$ is invertible at every $x\in D\smallsetminus
\{x\in\mathbb{R}^{m}:0<\alpha(x/\left\vert x\right\vert )<\left\vert
x\right\vert <\beta(x/\left\vert x\right\vert )\}$.

\item[(iv)] The radial derivative denoted as $\partial_{r}\left\vert
f\right\vert $ and defined as
\[
\partial_{r}\left\vert f\right\vert (x):=\left\langle \nabla\left\vert
f\right\vert (x),x/\left\vert x\right\vert \right\rangle ,
\]
when it exists, is such that there are numbers $\lambda,\mu$ with
$\mathfrak{M}/\mathfrak{m}<\lambda<\mu$ for which $\lambda\leq\partial
_{r}\left\vert f\right\vert (x)\leq\mu$ for every $x\in\{x\in\mathbb{R}%
^{m}:0<\left\vert x\right\vert \leq\alpha\left(  x/\left\vert x\right\vert
\right)  \}$ and $-\mu\leq\partial_{r}\left\vert f\right\vert (x)\leq-\lambda$
whenever $x\in\{x\in\mathbb{R}^{m}:\beta\left(  x/\left\vert x\right\vert
\right)  \leq\left\vert x\right\vert <\zeta\left(  x/\left\vert x\right\vert
\right)  \} $. Here $\mathfrak{m}:=\min\{\alpha(u):u\in\mathbb{S}%
^{m-1}\},\mathfrak{M}:=\max\{\beta(u):u\in\mathbb{S}^{m-1}\}$ and $\nabla$ is
the usual gradient operator.
\end{itemize}

\noindent Then
\begin{equation}
\Lambda:=\bar{D}\smallsetminus{\displaystyle\bigcup\nolimits_{n=1}^{\infty}%
}f^{-n}\left(  \mathring{S}_{\ast}\right)  ,\label{e3}%
\end{equation}
is a compact semichaotic strange global attracting set having $m$-dimensional
Lebesgue measure zero, where $\bar{E}$ and $\mathring{E}$ denote the closure
and interior, respectively, of the set $E$.
\end{theorem}

\begin{proof}
It follows directly from the hypotheses and construction that $\Lambda$ is a
compact global attracting set for $f$. Moreover, excepting the origin, it is
essentially a two-component Cantor set (generated by $(m-1)$-spherical shells
rather than intervals). Consequently, it must be a fractal set. More
precisely, we see from the properties (in particular (iv)) of the map that we
have%
\[
f^{-1}\left(  S_{\ast}\right)  =S_{0}\cup S_{1},
\]
where the union is disjoint and both $S_{0}$ and $S_{1}$ are open
$(m-1)$-spherical shells such that $S_{i}\subset\mathring{\Sigma}_{i}$,
$i=0,1$, where%
\[
\Sigma_{0}:=\{0\}\cup\{x\in\mathbb{R}^{m}:0<\left\vert x\right\vert \leq
\alpha\left(  x/\left\vert x\right\vert \right)  \}\text{ and }\Sigma
_{1}:=\{x\in\mathbb{R}^{m}:\beta\left(  x/\left\vert x\right\vert \right)
\leq\left\vert x\right\vert \leq\zeta\left(  x/\left\vert x\right\vert
\right)  \},
\]
and we note that we have the partition
\[
\bar{D}=\Sigma_{0}\cup S_{\ast}\cup\Sigma_{1}.
\]
Repeating the pre-imaging operation on the open $(m-1)$-spherical shells, we
obtain the partitions of open cells%
\[
f^{-1}\left(  S_{0}\right)  =S_{00}\cup S_{01}\text{ and }f^{-1}\left(
S_{1}\right)  =S_{10}\cup S_{11},
\]
where $S_{0i},S_{1i}\subset\mathring{\Sigma}_{0i}$, $i=0,1$, and the sets are
defined by the following partitions written in what is obviously meant by
`radial order'
\[
\Sigma_{0}=\Sigma_{00}\cup S_{0}\cup\Sigma_{01}\text{ and }\Sigma_{1}%
=\Sigma_{10}\cup S_{1}\cup\Sigma_{11}.
\]
But this is, if continued, just the standard inductive construction for the
Cantor set, so we obtain the disjoint union representation%
\[
\Lambda={\displaystyle\bigvee\nolimits_{s\in2^{\mathbb{N}}}}\Sigma_{s},
\]
where, as usual, $2^{\mathbb{N}}$ is the set of maps $s:\mathbb{N}%
\rightarrow\{0,1\}$, which can, of course be identified with the set of binary
sequences%
\[
\{.a_{1}a_{2}a_{3}\ldots:a_{i}=0\text{ or }1\text{ for all }i\in\mathbb{N}\}.
\]
Whence, it follows directly from a simple argument, based on sliding the shell
boundaries along rays, that $\Lambda$ is homeomorphic and actually $C^{1}$
diffeomorphic on the complement of the origin to a fractal `cone' pinched at
the origin; namely
\begin{equation}
\Lambda\cong\left(  \mathbb{S}^{m-1}\times C\right)  /\mathbb{S}^{m-1}%
\times0,\label{eq4}%
\end{equation}
where $\cong$ denotes homeomorphic, $C$ is a standard two-component Cantor set
on the unit interval $[0,1]$, and the space on the right above has the usual
quotient topology. For convenience, we shall refer to (\ref{eq4}) as the
\emph{Cantor cone of} $\ \mathbb{S}^{m-1}$. This conclusively shows that the
attracting set is fractal. As for the sensitive dependence on initial
conditions, this is an immediate consequence of the construction of $\Lambda$
and (iv). Indeed, we compute that for all $x\in D$
\begin{align*}
n^{-1}\log\left\Vert (f^{n})^{\prime}(x)\right\Vert  &  =n^{-1}%
{\displaystyle\sum\nolimits_{k=1}^{n}}\log\left\Vert f^{\prime}\left(
f^{k-1}(x)\right)  \right\Vert \\
&  \geq n^{-1}{\displaystyle\sum\nolimits_{k=1}^{n}}\log\left\vert
\partial_{r}\left\vert f\right\vert (f^{k-1}(x))\right\vert \geq\log\lambda>0,
\end{align*}
so $\lim\inf n^{-1}\log\left\Vert (f^{n})^{\prime}(x)\right\Vert >0$, and the
sensitive dependence is established, which completes the proof.
\end{proof}

\noindent Observe that the maps described in the above theorem all have what
one might call a \emph{degenerate snap-back repeller} at the origin
(\emph{c.f.} Marotto \cite{Mar} and Chen \& Aihara \cite{CA}), which might be
related to their chaotic nature in some as yet to be determined way.

We note here one can approximate the fractal dimension of the attractor in the
above theorem. Using standard fractal techniques (see \emph{e.g.}\cite{Falc}),
a straightforward but more laborious argument than we care to go into here
yields the following estimate for the Hausdorff dimension $\dim_{H} $:%
\begin{equation}
m-1+\log2/\log(1+\mu)\leq\dim_{H}\left(  \Lambda\right)  \leq m-1+\log
2/\log(1+\lambda).\label{eq5}%
\end{equation}

If we examine the proof of Theorem \ref{thm3.1}, more information on the
structure of the restriction of the map to the attracting set can be readily
obtained. First, we may recast the Cantor cone (\ref{eq4}) in the form%
\begin{equation}
\Lambda\cong\left(  \mathbb{S}^{m-1}\times2^{\mathbb{N}}\right)  /\left(
\mathbb{S}^{m-1}\times0\right)  ,\label{e6}%
\end{equation}
where $2^{\mathbb{N}}$ is given the metric topology generated by%
\[
d_{B}\left(  s,\tilde{s}\right)  :={\displaystyle\sum\nolimits_{n=1}^{\infty}%
}2^{-n}\left\vert s(n)-\tilde{s}(n)\right\vert .
\]
Then, by making fairly obvious modifications of a standard argument, used for
example in studying the Smale horseshoe via symbolic dynamics (see, e.g.
\cite{Shub}), it is a straightforward matter to verify the following result.

\begin{corollary}
\label{cor3.1} The hypotheses of Theorem \ref{thm3.1} implies that $f$ is
conjugate on $\Lambda$ to a map of the form
\[
\hat{f}:\left(  \mathbb{S}^{m-1}\times2^{\mathbb{N}}\right)  /\left(
\mathbb{S}^{m-1}\times0\right)  \rightarrow\left(  \mathbb{S}^{m-1}%
\times2^{\mathbb{N}}\right)  /\left(  \mathbb{S}^{m-1}\times0\right)  ,
\]
where
\[
\hat{f}(x,s):=\left(  \nu(x,s),\sigma(s)\right)  ,
\]
$\sigma$ is the shift map $($with $\sigma(.a_{1}a_{2}a_{3}\ldots)=(.a_{2}%
a_{3}a_{4}\ldots))$ and $\nu$ is a continuous map.
\end{corollary}

If we have additional information concerning the behavior of the map on the
spherical shell comprising the global attractor given by (\ref{e3}) in Theorem
\ref{thm3.1}, it is sometimes possible to prove that one has in fact a true
chaotic strange attractor. An example of such a result is the following, which
clearly can be readily generalized.

\begin{corollary}
\label{cor3.2} Suppose in addition to the hypotheses of Theorem \ref{thm3.1},
$f$ satisfies the following property: There exists a finite set of $C^{1}$
curves $\{\gamma_{1},\gamma_{2},\ldots,\gamma_{k}\}\subset\bar{D}$ such that:

\begin{itemize}
\item[(a)] Each curve begins at $x=0$ and ends at a distinct point of the
boundary $\partial\bar{D}=\{x\in\mathbb{R}^{m}:x=\zeta\left(  x/\left\vert
x\right\vert \right)  \}$ of $\bar{D}.$

\item[(b)] The curves are all transverse to $\partial S_{\ast}=$
$\{x\in\mathbb{R}^{m}:x=\alpha\left(  x/\left\vert x\right\vert \right)
\}\cup\{x\in\mathbb{R}^{m}:x=\beta\left(  x/\left\vert x\right\vert \right)
\}.$

\item[(c)] $f(\gamma_{j})=\gamma_{j+1}$, $1\leq j\leq k-1$, and $f(\gamma
_{k})=\gamma_{1}$.

\item[(d)] The set $E:=\gamma_{1}\cup\gamma_{2}\cup\cdots\cup\gamma_{k}$ is
\textbf{conically attracting} in the sense that there is a conical open set
(pinched at the origin) $W$ such that $d(f^{n}(x),E)\rightarrow0$ as
$n\rightarrow\infty$ whenever $x\in E.$
\end{itemize}

Then
\[
\mathfrak{A}:=\Lambda\cap E,
\]
where $\Lambda$ is as in (\ref{e3}), is a chaotic strange global attractor.
\end{corollary}

\begin{proof}
It follows at once from the hypotheses and Theorem \ref{thm3.1} that
$\mathfrak{A}$ is a semichaotic strange global attractor. Consequently, it
suffices to prove that $f$ restricted to $\mathfrak{A}$, denoted as $g:=$
$f|_{\mathfrak{A}}$, which can be identified with its conjugate described in
Corollary \ref{cor3.1}, is both topologically transitive and that its periodic
points are dense. We first prove that given any $x,\tilde{x}\in\mathfrak{A}$
and any open neighborhood $U$ of $x$, there is a point $\breve{x}\in U$ and an
$N\in\mathbb{N}$ such that $g^{N}(\breve{x})=\tilde{x}$, which clearly implies
topological transitivity. Let $x\in\gamma_{i}\cap\Lambda$ belong to the
$(m-1)$-spherical shell corresponding to $s=.a_{1}a_{2}\ldots$and $\tilde
{x}\in\gamma_{j}\cap\Lambda$ belong to the $(m-1)$-spherical shell
corresponding to $\tilde{s}=.\tilde{a}_{1}\tilde{a}_{2}\ldots$. Define $q$ to
be the least nonnegative integer such that $g^{q}(\gamma_{i})=\gamma_{j}$ and
note that it follows from (c) that $g^{lk+q}(\gamma_{i})=\gamma_{j}$ for every
nonnegative integer $l$. Given any $\epsilon>0$, we can obviously choose $l$
so large that $d_{B}(s,\breve{s})<\epsilon$, where
\[
\breve{s}:=.a_{1}a_{2}\ldots a_{lk+q}\tilde{s}:=.a_{1}a_{2}\ldots
a_{lk+q}\tilde{a}_{1}\tilde{a}_{2}\ldots,
\]
and then define $\breve{x}$ to be the intersection of $\gamma_{i}$ with the
$(m-1)$-spherical shell associated to $\breve{s}$. Clearly, by selecting
$\epsilon$ sufficiently small, we can guarantee that $\breve{x}\in U$. But it
follows from our construction that $\sigma^{lk+q}(\breve{s})=\tilde{s}$, so
$g^{lk+q}(\breve{x})=\tilde{x}$ and the topological transitivity follows.
The density of the periodic points can be readily verified by a straightforward
modification of the transitivity argument. To wit, let $x\in\gamma_{i}%
\cap\Lambda$ belong to the $(m-1)$-spherical shell corresponding to
$s=.a_{1}a_{2}\ldots$. Given any $\epsilon>0$, select $l$ so large that
$d_{B}(s,\overline{.a_{1}\ldots a_{lk}})<\epsilon$, where $\overline
{.a_{1}\ldots a_{lk}}$ is the binary sequence formed by successively
concatenating the finite sequence $.a_{1}a_{2}\ldots a_{lk}$ with itself,
\emph{i.e}.%
\[
\overline{.a_{1}\ldots a_{lk}}:=.a_{1}a_{2}\ldots a_{lk}a_{1}a_{2}\ldots
a_{lk}\ldots.
\]
It is easy to see that  $\overline{.a_{1}\ldots a_{lk}}$ is a periodic point
of period $lk$ for the shift map $\sigma$ and if we select $\bar{x}$ to be the
point of $\gamma_{i}$ on the $(m-1)$-spherical shell corresponding to
$\overline{.a_{1}\ldots a_{lk}}$, then $g^{lk}(\bar{x})=\bar{x}$, so the proof
is complete since $\epsilon$ is arbitrary.
\end{proof}

\subsection{Attractors for AZ maps contracting at the origin}

When the origin is a sink rather than a source, as it typically is for models
of populations comprised entirely of climax species, we can relax the
extermination requirement of Theorem \ref{thm3.1} and obtain analogous results
concerning strange attractor for globally $C^{1}$ maps. An example is the following.

\begin{theorem}
\label{thm3.2} Suppose $f:\mathbb{R}^{m}\rightarrow\mathbb{R}^{m}$ is a
$C^{1}$ EZ map, with $M$ and $R_{M}$ as in Lemma \ref{lem2.1.3}, for which the
following properties obtain:

\begin{itemize}
\item[(i)] $f^{-1}\left(  \{0\}\right)  =\{0\}$, $\left\Vert f^{\prime
}(0)\right\Vert <1$ and $\{0\}$ has a basin of attraction of the form%
\[
\mathcal{B}(0):=\{x\in\mathbb{R}^{m}:0\leq\left\vert x\right\vert <\alpha
_{0}(x/\left\vert x\right\vert )\},
\]
where $\alpha_{0}:\mathbb{S}^{m-1}\rightarrow\mathbb{R}$ is a positive $C^{1}
$ function and $f^{-1}\left(  \mathcal{B}(0)\right)  $ is a semi-infinite
$(m-1)$-spherical shell of the form%
\[
Z:=\{x\in\mathbb{R}^{m}:\zeta\left(  x/\left\vert x\right\vert \right)
<\left\vert x\right\vert \},
\]
where $\zeta:\mathbb{S}^{m-1}\rightarrow\mathbb{R}$ is a $C^{1}$ function
satisfying $R_{M}<\zeta(u)<M$ for all $u\in\mathbb{S}^{m-1}$.

\item[(ii)] The set $S_{\ast}:=f^{-1}\left(  \bar{Z}\right)  $ is an
$(m-1)$-spherical shell of the form%
\[
S_{\ast}=\{x\in\mathbb{R}^{m}:0<\alpha(x/\left\vert x\right\vert
)\leq\left\vert x\right\vert \leq\beta(x/\left\vert x\right\vert )\},
\]
where $\alpha,\beta:\mathbb{S}^{m-1}\rightarrow\mathbb{R}$ are positive
$C^{1}$ functions such that $0<\beta(u)-\alpha(u)<\zeta(u)$ for all
$u\in\mathbb{S}^{m-1}$.

\item[(iii)] $f^{\prime}(x)$ is invertible at every $x\in D\smallsetminus
\{x\in\mathbb{R}^{m}:0<\alpha(x/\left\vert x\right\vert )<\left\vert
x\right\vert <\beta(x/\left\vert x\right\vert )\}$.

\item[(iv)] The radial derivative denoted as $\partial_{r}\left\vert
f\right\vert $ and defined as
\[
\partial_{r}\left\vert f\right\vert (x):=\left\langle \nabla\left\vert
f\right\vert (x),x/\left\vert x\right\vert \right\rangle ,
\]
when it exists, is such that there are numbers $\lambda,\mu$ with
$\mathfrak{M}/\mathfrak{m}<\lambda<\mu$ for which $\lambda\leq\partial
_{r}\left\vert f\right\vert (x)\leq\mu$ for every $x\in\{x\in\mathbb{R}%
^{m}:0<\left\vert x\right\vert \leq\alpha\left(  x/\left\vert x\right\vert
\right)  \}$ and $-\mu\leq\partial_{r}\left\vert f\right\vert (x)\leq-\lambda$
whenever $x\in\{x\in\mathbb{R}^{m}:\beta\left(  x/\left\vert x\right\vert
\right)  \leq\left\vert x\right\vert <\zeta\left(  x/\left\vert x\right\vert
\right)  \} $. Here $\mathfrak{m}:=\min\{\alpha(u):u\in\mathbb{S}%
^{m-1}\},\mathfrak{M}:=\max\{\beta(u):u\in\mathbb{S}^{m-1}\}$ and $\nabla$ is
the usual gradient operator.
\end{itemize}

\noindent Then
\begin{equation}
\Gamma:=\{0\}\vee\Gamma_{C},\label{e7}%
\end{equation}
where
\[
\Gamma_{C}:=\bar{D}\smallsetminus{\displaystyle\bigcup\nolimits_{n=1}^{\infty
}}f^{-n}\left(  \mathring{S}_{\ast}\right)  ,
\]
is a compact semichaotic strange minimal global attracting set having
$m$-dimensional Lebesgue measure zero.
\end{theorem}

\begin{proof}
As the proof of this result is completely analogous to that of Theorem
\ref{thm3.1}, we need only sketch the argument. Again we immediately deduce
from the hypotheses and construction that $\Gamma$ is a compact global
attracting set for $f$. Moreover, excepting the origin, it is a two-component
Cantor set (generated by $(m-1)$-spherical shells rather than intervals)
contained in $\{x\in\mathbb{R}^{m}:\alpha_{0}(x/\left\vert x\right\vert
)\leq\left\vert x\right\vert \leq\zeta(x/\left\vert x\right\vert )\}$.
Consequently, it must be a fractal set. More precisely, we see from the
properties (in particular (iv)) of the map that we have%
\[
f^{-1}\left(  S_{\ast}\right)  =S_{0}\cup S_{1},
\]
where the union is disjoint and both $S_{0}$ and $S_{1}$ are open
$(m-1)$-spherical shells such that $S_{i}\subset\mathring{\Sigma}_{i}$,
$i=0,1$, where%
\[
\Sigma_{0}:=\{0\}\cup\{x\in\mathbb{R}^{m}:\alpha_{0}\left(  x/\left\vert
x\right\vert \right)  \leq\left\vert x\right\vert \leq\alpha\left(
x/\left\vert x\right\vert \right)  \}\text{ and }\Sigma_{1}:=\{x\in
\mathbb{R}^{m}:\beta\left(  x/\left\vert x\right\vert \right)  \leq\left\vert
x\right\vert \leq\zeta\left(  x/\left\vert x\right\vert \right)  \},
\]
and we note that we have the partition
\[
\bar{D}=\Sigma_{0}\cup S_{\ast}\cup\Sigma_{1}.
\]
Repeating the pre-imaging operation on the open $(m-1)$-spherical shells, we
obtain the partitions of open cells%
\[
f^{-1}\left(  S_{0}\right)  =S_{00}\cup S_{01}\text{ and }f^{-1}\left(
S_{1}\right)  =S_{10}\cup S_{11},
\]
where $S_{0i},S_{1i}\subset\mathring{\Sigma}_{0i}$, $i=0,1$, and the
continuing, we obtain in a manner quite like that in the proof of Theorem
\ref{thm3.1} the Cantor set representation%
\[
\Gamma_{C}={\displaystyle\bigvee\nolimits_{s\in2^{\mathbb{N}}}}\Sigma
_{s}\subset\bar{D}.
\]
Whence, by again sliding the shell boundaries along rays, we see that
$\Gamma_{C}$ is homeomorphic and actually $C^{1}$ diffeomorphic with
\begin{equation}
\Gamma_{C}\cong\mathbb{S}^{m-1}\times C,
\end{equation}
where $\cong$ denotes homeomorphic and $C$ is a standard two-component Cantor
set on the unit interval $[0,1]$. This conclusively shows that the attracting
set $\{0\}\cup$ $\Gamma_{C}$ is fractal. The sensitive dependence is an
immediate consequence of the construction of $\Gamma_{C}$ and (iv). Indeed, we
compute that for all $x\in D$
\begin{align*}
n^{-1}\log\left\Vert (f^{n})^{\prime}(x)\right\Vert  &  =n^{-1}%
{\displaystyle\sum\nolimits_{k=1}^{n}}\log\left\Vert f^{\prime}\left(
f^{k-1}(x)\right)  \right\Vert \\
&  \geq n^{-1}{\displaystyle\sum\nolimits_{k=1}^{n}}\log\left\vert
\partial_{r}\left\vert f\right\vert (f^{k-1}(x))\right\vert \geq\log\lambda>0,
\end{align*}
so $\lim\inf n^{-1}\log\left\Vert (f^{n})^{\prime}(x)\right\Vert >0$, which
implies sensitive dependence, and the proof is complete.
\end{proof}

The approximate value for the Hausdorff dimension (\ref{eq5}) also applies
here, and there are rather obvious analogs (where the origin is a sink rather
than a source) of Corollaries \ref{cor3.1} and \ref{cor3.2} for Theorem
\ref{thm3.2}. We only state these results, since their proofs are essentially
the same \emph{mutatis mutandis} as those given for their analogs.

\begin{corollary}
\label{cor3.2.1} The hypotheses of Theorem \ref{thm3.2} implies that $f$ is
conjugate on $\Gamma$ to a map of the form
\[
\hat{f}:\left[  \{0\}\vee\left(  \mathbb{S}^{m-1}\times2^{\mathbb{N}}\right)
\right]  /\left(  \mathbb{S}^{m-1}\times0\right)  \rightarrow\left[
\{0\}\vee\left(  \mathbb{S}^{m-1}\times2^{\mathbb{N}}\right)  \right]
/\left(  \mathbb{S}^{m-1}\times0\right)  ,
\]
where
\[
\hat{f}(0):=0\text{ and }\hat{f}(x,s):=\left(  \nu(x,s),\sigma(s)\right)
\text{ for }x\neq0,
\]
$\sigma$ is the shift map $($with $\sigma(.a_{1}a_{2}a_{3}\ldots)=(.a_{2}%
a_{3}a_{4}\ldots))$ and $\nu$ is a continuous map.
\end{corollary}

\begin{corollary}
\label{cor3.2.2} Suppose in addition to the hypotheses of Theorem
\ref{thm3.2}, $f$ satisfies the following property: There exists a finite set
of $C^{1}$ curves $\{\gamma_{1},\gamma_{2},\ldots,\gamma_{k}\}\subset\bar{D}$
such that:

\begin{itemize}
\item[(a)] Each curve begins at $x=0$ and ends at a distinct point of the
boundary $\partial\bar{D}=\{x\in\mathbb{R}^{m}:x=\zeta\left(  x/\left\vert
x\right\vert \right)  \}$ of $\bar{D}.$

\item[(b)] The curves are all transverse to $\partial S_{\ast}=$
$\{x\in\mathbb{R}^{m}:x=\alpha\left(  x/\left\vert x\right\vert \right)
\}\cup\{x\in\mathbb{R}^{m}:x=\beta\left(  x/\left\vert x\right\vert \right)
\}.$

\item[(c)] $f(\gamma_{j})=\gamma_{j+1}$, $1\leq j\leq k-1$, and $f(\gamma
_{k})=\gamma_{1}$.

\item[(d)] The set $E:=\gamma_{1}\cup\gamma_{2}\cup\cdots\cup\gamma_{k}$ is
conically attracting.
\end{itemize}

Then
\[
\mathfrak{A}:=\Gamma\cap E,
\]
where $\Gamma$ is as in (\ref{e7}), is a chaotic strange minimal global attractor.
\end{corollary}

A few remarks are in order before we move on to a discussion of another - very
different - type of strange attractor. In Theorem \ref{thm3.2}, the origin is
actually a global metric attractor (see, \emph{e.g}. Milnor \cite{Mil1,Mil2})
because $\Gamma_{C}$ has $m$-dimensional Lebesgue measure zero, but the
minimal global attractor is indeed $\Gamma$. The basin of attraction is
riddled in a measure zero sense (\emph{c.f}. \cite{Caz}). In applications, the
model maps frequently leave all of the coordinate axes invariant (such as in
\cite{CB,Has,JB1,JB2,JF,Seln,Sum}) in which case there are obvious versions of
the above strange attractor results for maps restricted to $x_{1},\ldots
,x_{m}\geq0$, for example.

\section{Multihorseshoe (Trellis) Strange Attractors}

In this section we describe another type of attractor generated by one or more
horseshoes associated to hyperbolic fixed or periodic points. These attractors
apparently were first studied in reasonable detail (for planar maps with an
emphasis on characterizing the stable-unstable manifold connections of
neighboring horseshoes) by Easton \cite{Easton}, who seems also to have coined
the name \textquotedblleft trellis\textquotedblright\ to describe their
structure. They can often be found in the dynamics of AZ maps, but they may
occur for more general smooth maps as well. It is convenient to introduce the
following concept.

\begin{definition}
\label{defAH} Let $f:X\rightarrow X$ be a $C^{1}$ self map of the
$m$-dimensional $C^{1}$ manifold $X$. A subset $H$ of $X$ is an
\textbf{attracting 1}$\boldsymbol{\times}$\textbf{(m-1)-horseshoe(at p)} for
$f$ if the following properties obtain:

\begin{itemize}
\item[]
\begin{itemize}
\item[(AH1)] There is a $C^{1}$ diffeomorphism $\varphi:H^{\ast}\rightarrow H
$, where
\[
H^{\ast}:=C_{0}^{\ast}\cup Z^{\ast}\cup C_{1}^{\ast},
\]
with%
\begin{align*}
Z^{\ast}  &  :=\left\{  x\in\mathbb{R}^{m}:0\leq(x_{1}+2)^{2}+x_{2}^{2}%
+\cdots+x_{m-1}^{2}\leq4^{2},\;-1\leq x_{m}\leq9\right\}  ,\\
C_{0}^{\ast}  &  :=\left\{  x\in\mathbb{R}^{m}:0\leq(x_{1}+2)^{2}+x_{2}%
^{2}+\cdots+(x_{m}+1)\leq4^{2},\;x_{m}\leq-1\right\}  \text{ and}\\
C_{1}^{\ast}  &  :=\left\{  x\in\mathbb{R}^{m}:0\leq(x_{1}+2)^{2}+x_{2}%
^{2}+\cdots+(x_{m}-9)^{2}\leq4^{2},\;x_{m}\geq9\right\}  ;
\end{align*}
so that defining $H:=\varphi\left(  H^{\ast}\right)  $, $Z:=\varphi\left(
Z^{\ast}\right)  $ and $C_{k}:=\varphi\left(  C_{k}^{\ast}\right)  $, $k=0,1$,
we have the associated decomposition%
\[
H=C_{0}\cup Z\cup C_{1}.
\]
We note that $H^{\ast}$ has both a $C^{\infty}$1-dimensional foliation
$\mathcal{F}_{\mathrm{v}}^{\ast}$ comprised of leaves that are the
intersections with $H^{\ast}$ of the lines of the form $x_{1},\ldots,x_{m-1}$
all constant, and a $C^{\infty}(m-1)$-dimensional foliation $\mathcal{F}%
_{\mathrm{h}}^{\ast}$ consisting of the intersections with $H^{\ast}$ of the
hyperplanes $x_{m}=$ constant. These orthogonal foliations produce transverse
$C^{1}$ foliations of $H$ defined by $\mathcal{F}_{\mathrm{v}}:=\varphi\left(
\mathcal{F}_{\mathrm{v}}^{\ast}\right)  $ and $\mathcal{F}_{\mathrm{h}%
}:=\varphi\left(  \mathcal{F}_{\mathrm{h}}^{\ast}\right)  $.

\item[(AH2)] The map $f$ restricted to $H$ is injective, $f\left(  H\right)
\subset\mathring{H}$, and $f\left(  C_{k}\right)  \subset\mathring{C}_{0}$ for
$k=0,1$.

\item[(AH3)] The foliation $f\left(  \mathcal{F}_{\mathrm{v}}\right)  $ is
transverse to the foliation $\mathcal{F}_{\mathrm{h}}$ in $Z$.

\item[(AH4)] $f$ restricted to $C_{0}$ is a contraction mapping into
$\mathring{C}_{0}$, and the unique fixed point $q$ in $\mathring{C}_{0}$ is an
attractor with a basin of attraction containing $C_{0}$.

\item[(AH5)] There is a further decomposition of $Z^{\ast}$ defined as
\[
Z^{\ast}=S_{0}^{\ast}\cup S_{1/2}^{\ast}\cup S_{1}^{\ast},
\]
where $S_{0}^{\ast}:=\{x\in Z^{\ast}:-1\leq x_{m}\leq3\}$, $S_{1/2}^{\ast
}:=\{x\in Z^{\ast}:3\leq x_{m}\leq5\}$ and $S_{0}^{\ast}:=\{x\in Z^{\ast
}:5\leq x_{m}\leq9\}$, which naturally induces the decomposition of $Z$ given
by
\[
Z=S_{0}\cup S_{1/2}\cup S_{1},
\]
where $S_{r}:=\varphi\left(  S_{r}^{\ast}\right)  $, $r=0,1/2,1$. For this
decomposition, $f\left(  S_{1/2}\right)  \subset\mathring{C}_{1}$.

\item[(AH6)] There is a hyperbolic fixed point (saddle) $p\in f\left(
S_{0}\right)  \cap S_{0}$ , corresponding to $\varphi\left(  0\right)  ,$with
a $1-$dimensional unstable manifold $W^{u}(p)$ that is tangent at $p$ to the
image under $\varphi$ of the leaf of $\mathcal{F}_{\mathrm{v}}$ through $p$
and an $(m-1)-$dimensional stable manifold $W^{s}(p)$ that is tangent at $p$
to the image under $\varphi$ of the leaf of $\mathcal{F}_{\mathrm{h}}$ through
this point.

\item[(AH7)] $f$ restricted to $Z$ is a contracting map along the leaves of
$\mathcal{F}_{\mathrm{h}}$ with contraction coefficient $0<\lambda<1$, and an
expanding map along the leaves of $f\left(  \mathcal{F}_{\mathrm{v}}\right)  $
with expansion coefficient $1<\mu$.
\end{itemize}
\end{itemize}
\end{definition}

\noindent This rather verbose description of the criteria for the existence of
attracting horseshoes is actually quite easy to apply and even easier to
illustrate as shown in Fig. 1. With this nomenclature, we can now efficiently
proceed to one of our main results on strange attractors, the proof of which
includes a fairly straightforward extension of an argument used by Easton
\cite{Easton}.

\begin{theorem}
\label{thm4.1} Let $f:E\rightarrow E$ be a $C^{1}$ self-map of a connected
open subset of $\mathbb{R}^{m}$. If $f$ has an attracting $1\times(m-1)$
-horseshoe $H$ at $p\in E$, then%
\[
\mathfrak{A}:=\bigcap\nolimits_{n=1}^{\infty}f^{n}\left(  H\right)
=\overline{W^{u}(p)}%
\]
is a chaotic strange attractor of $f$ with a basin of attraction containing
$H$, and $\mathfrak{A}$ is homeomorphic to
\begin{equation}
\left(  K\times\lbrack0,1]\right)  /\left(  K\times\{0,1\}\right)  ,\label{e8}%
\end{equation}
where $K$ is a two-component Cantor space and the usual quotient topology is
used for the whole space.
\end{theorem}

\begin{proof}
We first verify that $\mathfrak{A}$ is an attracting set. If $x\in
H\smallsetminus\left(  S_{0}\cup S_{1}\right)  $, it follows from (AH2), (AH4)
and (AH5) that $f^{2}(x)\in\mathring{C}_{0}$. Consequently, in virtue of
(AH4), $f^{n}(x)\in\mathring{C}_{0}$ for all $n\geq2$ and $f^{n}(x)\rightarrow
q\in$ $\bigcap\nolimits_{n=1}^{\infty}f^{n}\left(  H\right)  =\overline
{W^{u}(p)}$ as $n\rightarrow\infty$. Moreover, we readily infer that the only
possible points $x$ in $H$ having iterates that do not converge to
$\mathfrak{A}$ are those such that $f^{n}(x)\in S_{0}\cup S_{1}$ for all
$n\geq0$. But for points such as these, it follows from (AH3), (AH6) and (AH7)
that $d\left(  f^{n}(x),\mathfrak{A}\right)  \leq c\lambda^{n}$ for all
$n\in\mathbb{N}$, so $f^{n}(x)\rightarrow\mathfrak{A}$. Thus, $\mathfrak{A}$
is an attracting set that includes all of $H$ in its basin of attraction. In
fact, the obvious fact that $\mathfrak{A}$ is indecomposable - which is probably most
self-evident from its characterization as the closure of the unstable manifold
of $p$ - shows that we have a minimal attracting set, which means that it is
an attractor.
To see that $\mathfrak{A}$ is fractal, one merely has to note that owing to
(AH3) and the geometry of $f(H),$ its intersection with $Z$ is homeomorphic
with $K\times\lbrack0,1]$, where $K$ is the two-component Cantor set in the
statement of the theorem.
As for the chaotic regimes of $f$ restricted to $\mathfrak{A}$, consider the
subset defined as
\[
\Lambda:=\mathfrak{A}\cap\left\{  x\in H:f^{n}(x)\in S_{0}\cup S_{1}\right\}
.
\]
This is just the standard product of Cantor sets $K\times K$, which is
$f$-invariant and on which it is well known (\emph{c.f}.
\cite{Arr,Gh,Rob,Shub,Wig} and Theorem \ref{thm3.1} and Corollary
\ref{cor3.1}) that $f$ is conjugate to the shift map $\sigma:2^{\mathbb{Z}%
}\rightarrow2^{\mathbb{Z}}$ defined as
\[
\sigma\left(  \{\cdots a_{-2}a_{-1}a_{0}.a_{1}a_{2}\cdots\}\right)  =\{\cdots
a_{-2}a_{-1}a_{0}a_{1.}a_{2}\cdots\}.
\]
This map is topologically transitive and the periodic points are dense, so it
follows that the map depends sensitively on initial points in $\Lambda$. The
sensitive dependence can also be established directly from the properties of
the attracting horseshoe; in fact, (AH6) and (AH7) implies that%
\[
\lim\inf{}_{n}n^{-1}\log\left\Vert (f^{n})^{\prime}(x)\right\Vert \geq\log\lambda>0
\]
for all $x\in\Lambda$. It remains only to prove the representation of the
attractor as the quotient space (\ref{e8}), which is obvious from the
definition of the attracting horseshoe, so the proof is complete.
\end{proof}

\noindent It is likely that rank-one theory could also be used to prove the
above theorem, but probably with considerably more effort than required for
our proof. In addition, it appears that our approach might lead to some
interesting generalizations that are beyond the scope of rank-one techniques.

\begin{figure}[ptb]
\begin{center}
\includegraphics[width=0.5\textwidth]{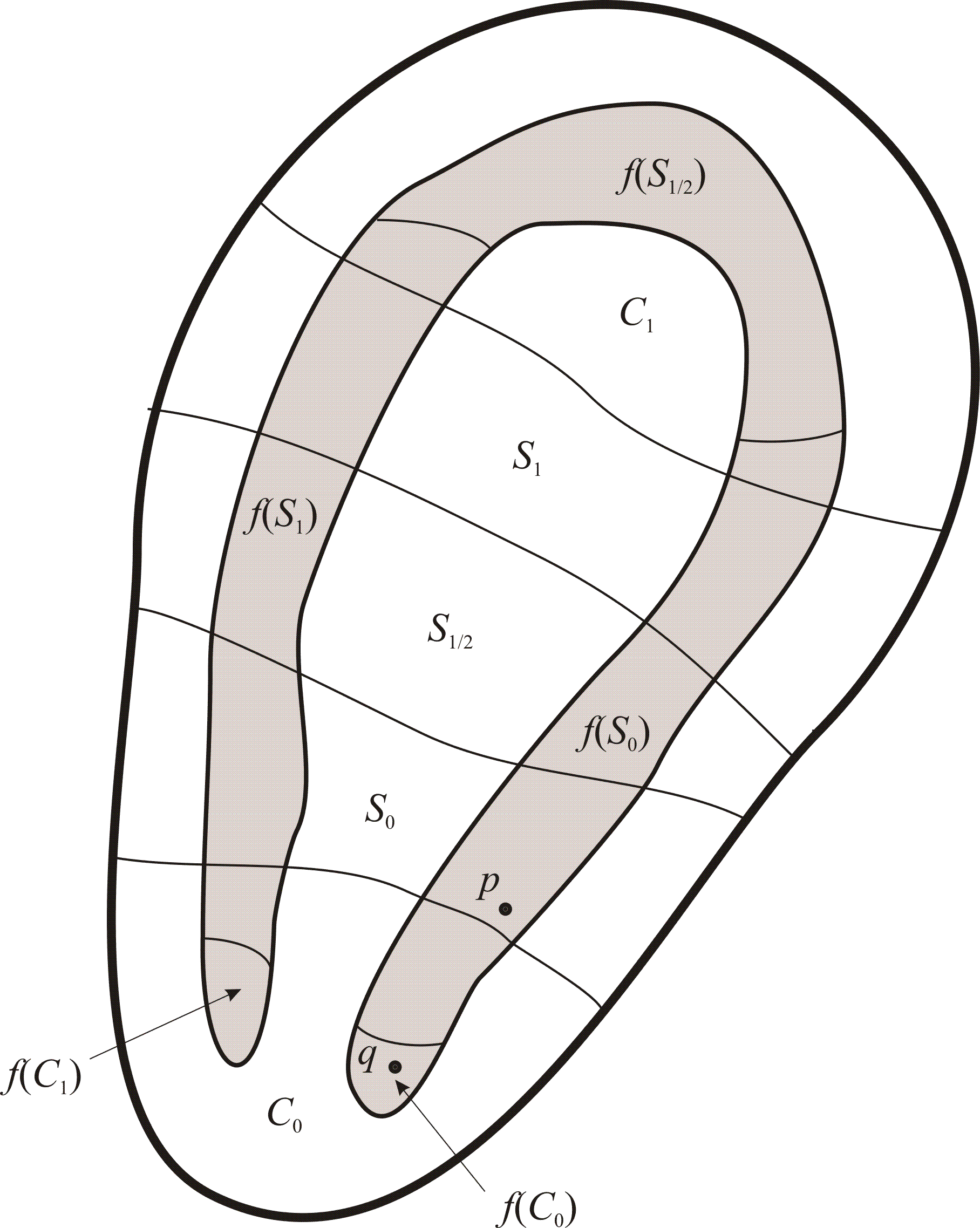}
\end{center}
\caption{An attracting horseshoe}%
\label{fig:H}%
\end{figure}

It is now an easy matter to generalize Theorem \ref{thm4.1} to the case of
multihorseshoe strange chaotic attractors (trellises), which tend to exemplify
both great complexity and aesthetically interesting patterns, which we shall
illustrate in the sequel. As the proof of the following results follows
directly from Theorem \ref{thm4.1}, we leave it to the reader.

\begin{theorem}
\label{thm4.2} If $f:E\rightarrow E$ a $C^{1}$ self-map of a connected open
subset of $\mathbb{R}^{m}$ that has a $k$-cycle of distinct points
$\{p_{1},\ldots,p_{k}\}$ with $k>1$ and $f^{k}$ has an attracting
$1\times(m-1)$ -horseshoe $H$ at one of the points $p$ in the cycle, then%
\[
\mathfrak{A}:=\overline{W^{u}(p)}\cup f\left(  \overline{W^{u}(p)}\right)
\cup\cdots\cup f^{k-1}\left(  \overline{W^{u}(p)}\right)
\]
is a chaotic strange attractor for $f$ with a basin of attraction containing
$H\cup f\left(  H\right)  \cup\cdots\cup f^{k-1}\left(  H\right)  .$
\end{theorem}

We remark that it is not difficult to find an analog of (\ref{eq5}) to
approximate the Hausdorff dimension of a single attracting horseshoe, from
which one can deduce that the fractal dimension is just slightly larger than
one if the contraction constant $\lambda$ in (AH7) is very small. On the other
hand, even for relatively small contraction constants, interactions of the
horseshoes (via intersections of respective unstable with stable manifolds)in
trellises can produce strange attractors that appear to have fractal
dimensions that are nearly equal to two for planar maps, as we shall
illustrate in the simulation examples to be shown in the next section.

\section{Applications and Examples}

We shall now illustrate our strange attractor theorems via simulation. Our
examples are chosen from well established discrete dynamical models for
physical phenomena; in particular, those for predicting the evolution of
ecological systems. For ease and clarity of visualization, we restrict the
examples to maps of the plane. We begin our simulations with maps that have
radial attractors

\subsection{Radial attractor examples}

Our first example involves the map $f:\mathbb{R}^{2}\rightarrow\mathbb{R}^{2}$
defined as
\begin{equation}
f(x_{1},x_{2})=f(x_{1},x_{2};a):=ae^{-x_{1}^{2}-x_{2}^{2}}\left(  x_{1}%
\cos(2\pi\theta)-x_{2}\sin(2\pi\theta),x_{1}\cos(2\pi\theta)-x_{2}\sin
(2\pi\theta)\right)  ,\label{e10}%
\end{equation}
where $\theta$ is an irrational number (so that the rotational part is
ergodic), is meant as an application of Theorem \ref{thm3.1}, but note that
although this map is AZ, it is not EZ. This particular discrete dynamical
system is of a type that has proven quite successful in modeling the evolution
of several pioneer species cohabiting and competing in the same ecological
environment. Fairly thorough descriptions of the properties of pioneer and
climax species can be found in \cite{JEAA,Frank,JB1,JF,Seln,Sum}.

Examples of radial type attractors for (\ref{e10}) that emerge for increasing
values of $a$ for two different choices of $\theta$ (the golden mean
$\phi:=(1+5)/2$ and the base of the natural logarithm $e$) are shown in Figs.
2 and 3. Reading the panels from left to right in these representations, the
parameter $a$ starts at $2.7$ in the upper left-hand corner and increases by
increments of $0.3$ until it reaches $6.0$ in the lower right-hand corner

Note how in each case the attractor essentially begins as a single invariant
ellipse, then there is period-doubling to a two-cycle of ellipses that begins
a period-doubling cascade that eventually leads to full-blown chaotic
attractors when $a$ is sufficiently large (around $a=4.2$). For larger values
of $a$ there is a parameter window (approximately $4.6<a<5$ ) in which the
attractor appears to be regular, then for all larger values of the parameter
($a\geq5.1$), one sees chaotic strange attractors of the kind described in
Theorem \ref{thm3.1}.

\begin{figure}[ptb]
\hspace{-0.4in}\includegraphics[width=1.1\textwidth]{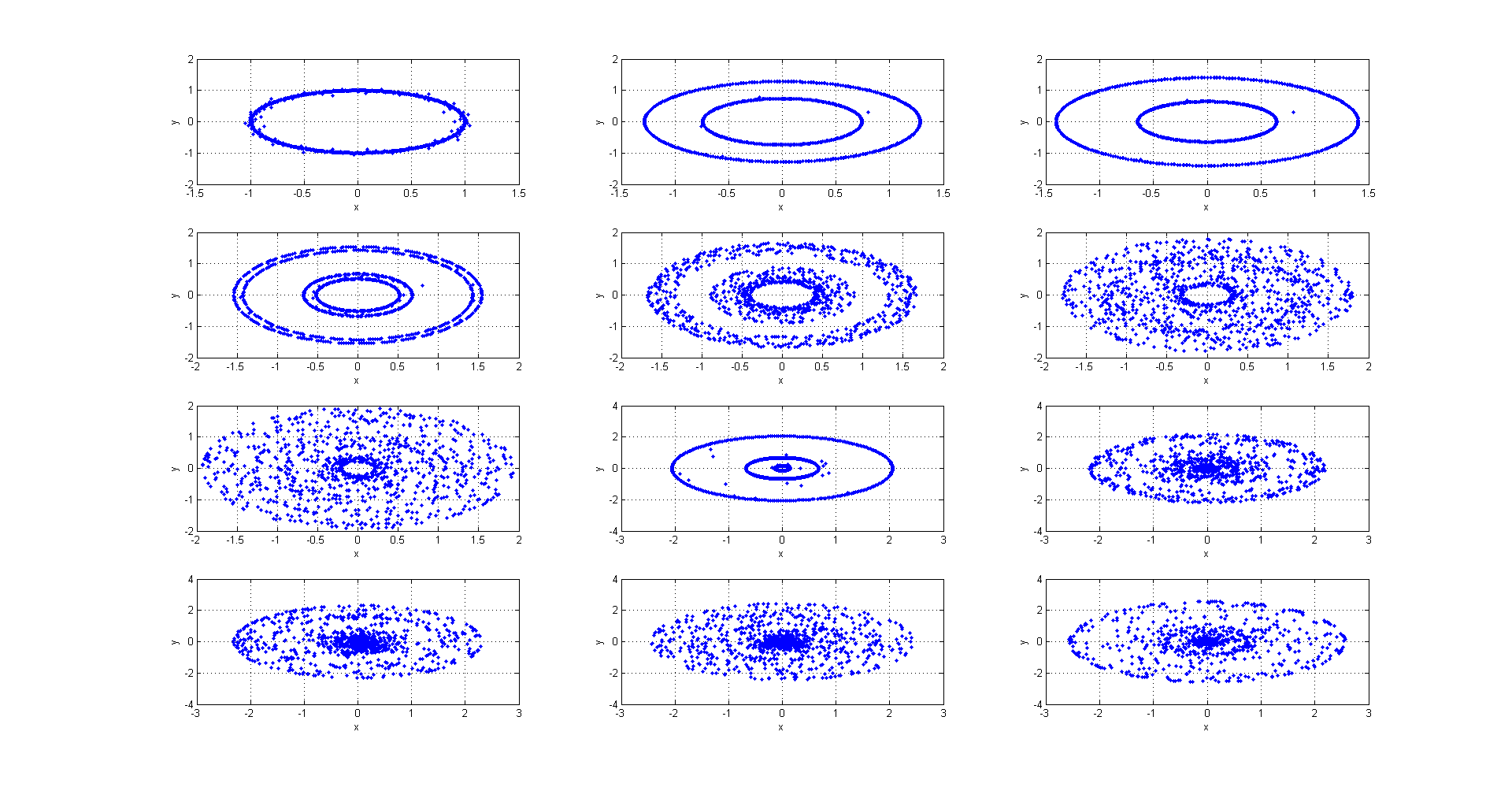}
\caption{Changes in radial attractors for (\ref{e10}) with
$\theta=\phi$ for $a=2.7,3,3.3,3.6,3.9,4.2,4.5,4.8,5.1,5.4,5.7,6.$}%
\label{fig:R1}%
\end{figure}

\begin{figure}[pbt]
\hspace{-0.4in}\includegraphics[width=1.1\textwidth]{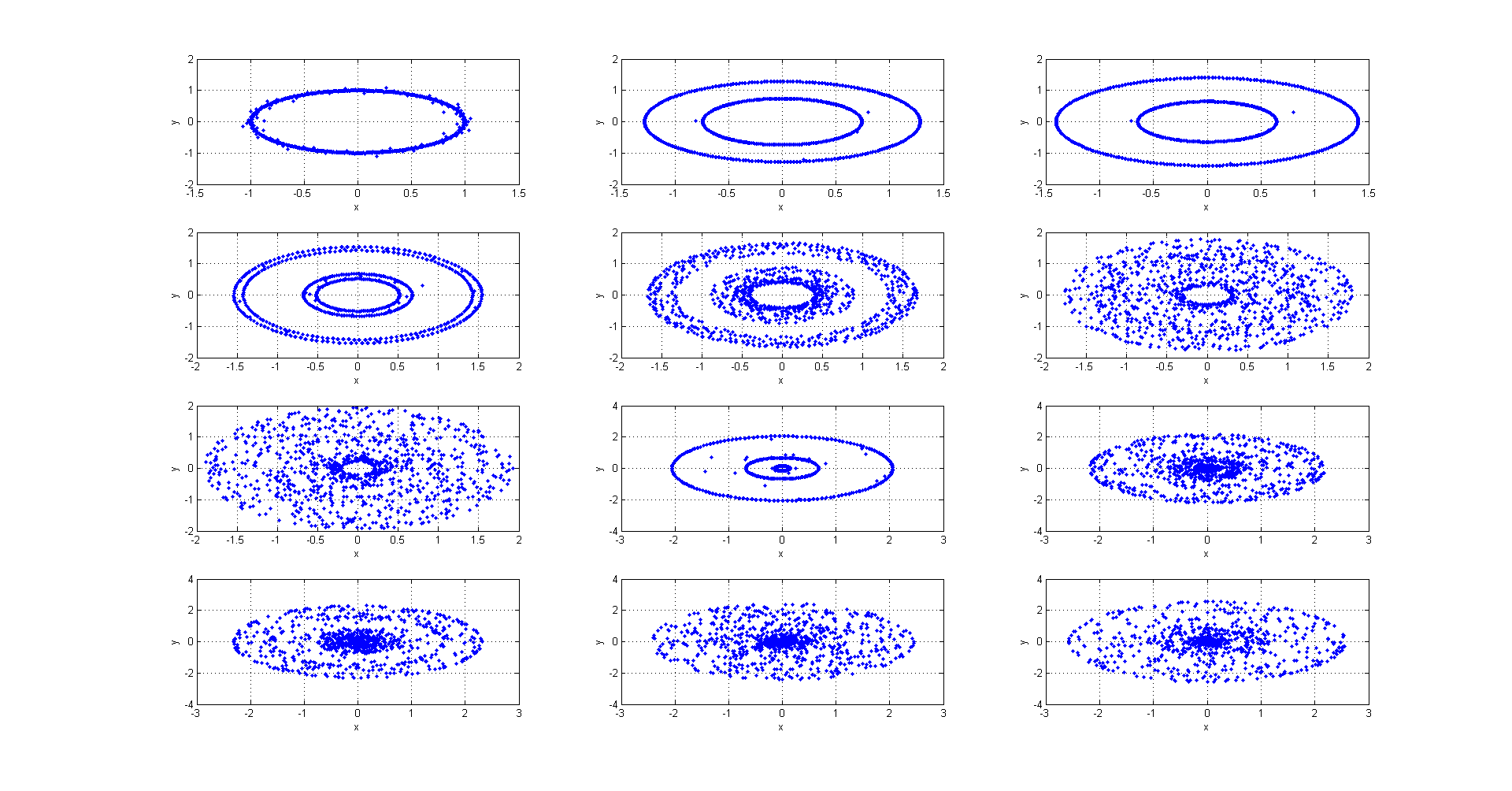}
\caption{Changes radial attractors for (\ref{e10}) with $\theta=e$
for $a=2.7,3,3.3,3.6,3.9,4.2,4.5,4.8,5.1,5.4,5.7,6.$}%
\label{fig:R2}%
\end{figure}

\subsection{Multihorseshoe attractor simulations}

Next, we illustrate the case in which there are multihorseshoe attractors for
different parameter values and note that the examples are drawn from models of
pioneer-climax ecological dynamics. For our first example of this kind, the
map $f:\mathbb{R}^{2}\rightarrow\mathbb{R}^{2}$ is
\begin{equation}
f(x_{1},x_{2})=f(x_{1},x_{2};a,b):=\left(  x_{1}e^{a-0.8x_{1}-0.2x_{2}}%
,x_{2}(0.2x_{1}+0.8x_{2})e^{b-0.2x_{1}-0.8x_{2}}\right)  .\label{e11}%
\end{equation}
Figure 4 shows the map (\ref{e11}) for $a=2.4$ and $b=2.5$ that follows
intervals of slightly smaller values for which $f$ has a $6$-cycle of sinks,
which develop nearby saddles and associated attracting horseshoes when $a$ and
$b$ are sufficiently large, as illustrated in Fig. 5 for $a=b=3$.

\begin{figure}[ptb]
\hspace{-0.4in}\includegraphics[width=1.1\textwidth]{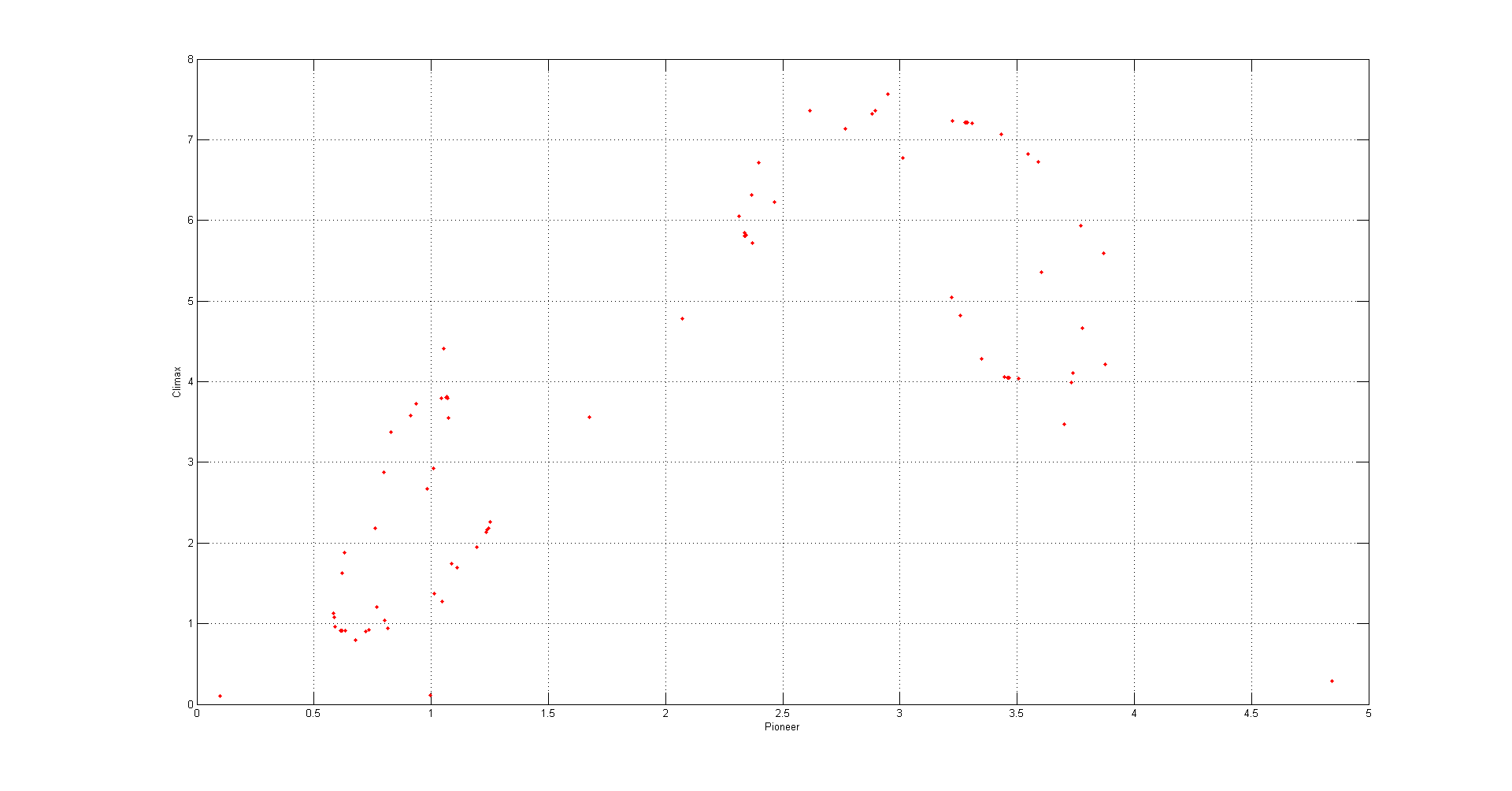}
\caption{Attractor of transitional multihorseshoe type for the map (\ref{e11})
with $a=2.4,b=2.5$.}%
\label{fig:MH1}%
\end{figure}

\begin{figure}[ptb]
\hspace{-0.4in}\includegraphics[width=1.1\textwidth]{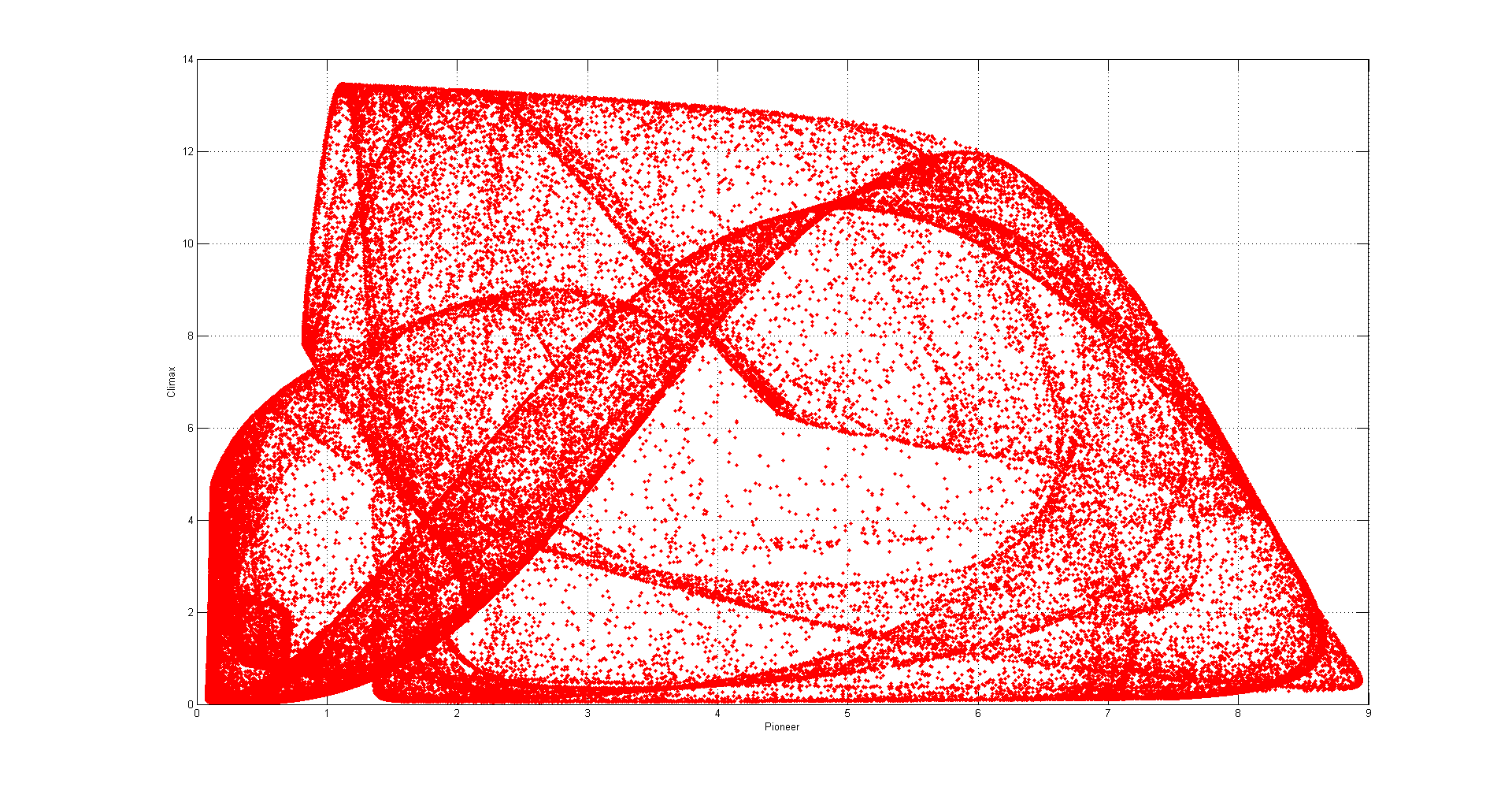}
\caption{A multihorseshoe chaotic strange attractor for the map (\ref{e11})
with $a=b=3$.}%
\label{fig:MH2}%
\end{figure}

Another example of a multihorseshoe attractor, this time for the map defined
as
\begin{equation}
f(x_{1},x_{2})=f(x_{1},x_{2};a,b):=\left(  x_{1}e^{a-0.8x_{1}},x_{2}%
(0.2x_{1}+0.8x_{2})e^{b-0.2x_{1}-0.8x_{2}}\right)  .\label{e12}%
\end{equation}
is shown in Fig. 6 for $a=b=3.$

\begin{figure}[ptb]
\hspace{-0.4in}\includegraphics[width=1.1\textwidth]{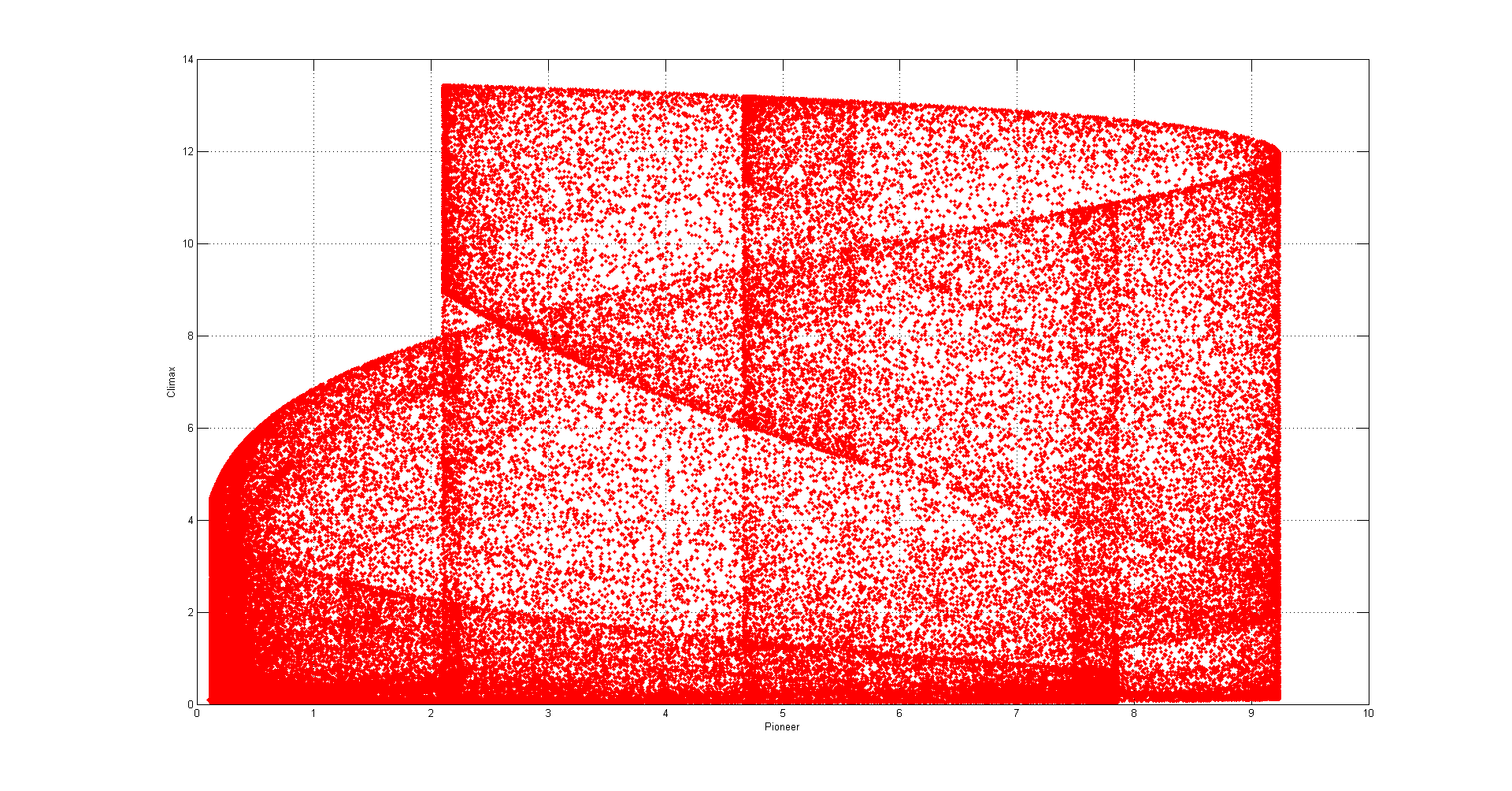}
\caption{A multihorseshoe chaotic strange attractor for the map (\ref{e12})
with $a=b=3$.}%
\label{fig:MH3}%
\end{figure}

\section{Concluding Remarks}

We have proved what appear to be several new theorems about strange attractors
based on hypotheses that are relatively easy to check, which means that they
are rather well suited to a variety of applications amenable to modeling by
discrete dynamical systems.

The radial theorems, which apply to asymptotically zero discrete dynamical
systems of arbitrary finite dimension, may be viewed as extensions of known
results for one-dimensional systems, since it appears that they pretty much
coincide with many aspects of much of the work that has appeared in the
literature in this case such as that described in \cite{BKNS,Thun,UW}. Our
multihorseshoe (trellis) results, even though they were at least partially
envisaged and analyzed by Easton, can still in large measure be considered as
novel inasmuch as they are rigorous extensions of the concepts introduced
in \cite{Easton}. As we also noted, the multihorseshoe attractors are
essentially of rank-one type, so it might be interesting to determine how
closely our approach is connected with rank-one theory. In addition, an
investigation of this connection might lead to some useful new techniques for
studying strange attractors that effectively combine elements of both methods.

One of the most attractive features of our multihorseshoe techniques is the
possibility they seem to provide for higher dimensional generalizations (with
may even extend to infinite-dimensional discrete dynamical systems). In
particular, it is natural to enquire if our approach can be used for the
construction and identification of chaotic strange attractors generated by
unstable manifolds of dimension greater that one. We have already begun to
look into this question, and expect to present rigorous versions of our
already promising preliminary results in a series of forthcoming papers. In
addition to extending and generalizing the results obtained here, we shall
continue to seek significant application areas in which our approach or some
modifications thereof can be used to answer outstanding questions concerning
the existence of strange attractors, which in physical modeling often provide
important information about the evolution of a system. For example, recently
developed dynamical models of granular flow phenomena such as in
\cite{BRTUZ,RBTUZ} appear to have strange attractors that control the
long-time behavior of the particle configurations.

\section*{Acknowledgment}

\noindent Y. Joshi would like to thank his department for support of his work
on this paper, and D. Blackmore is indebted to NSF Grant CMMI 1029809 for
partial support of his efforts in this collaboration.


\begin{thebibliography}{99}                                                                                               %
\bibitem {Arr}D. K. Arrowsmith and C. M. Place, \emph{Dynamical Systems:
Differential Equations, Maps and Chaotic Behaviour}, Chapman and Hall, London, 1992.

\bibitem {BBK}S. Balint, L. Braescu and E. Kaslik, \emph{Regions of Attraction
and Applications to Control Theory}, Cambridge Scientific Publ. Ltd.,
Cambridge, 2008.

\bibitem {BBCDS}J. Banks, J. Brooks, G. Cairns, G. Davis and P. Stacey, On
Devaney's definition of chaos, \emph{Amer. Math. Monthly} \textbf{99} (1992), 332-334.

\bibitem {BC}M. Benedicks and L. Carleson, The dynamics of the H\'{e}non map,
\emph{Ann. of Math. (2)} \textbf{133} (1991), 73-169.

\bibitem {Best}J. Best, C. Castillo-Chavez and A.-A. Yakubu, Hierarchical
competition in discrete time models with dispersal, \emph{Fields Institute
Communications} \textbf{36} (2004), 59-86.

\bibitem {BRTUZ}D. Blackmore, A. Rosato, X. Tricoche, K. Urban and L. Zuo,
Analysis, simulation and visualization of 1D tapping dynamics via reduced
dynamical models (submitted).

\bibitem {BKNS}H. Bruin, G. Keller, T. Nowicki and S. van Strien, Wild Cantor
attractors exist, \emph{Ann. of Math.} \textbf{143} (1996), 97-130.

\bibitem {Caz}B. Cazelles, Dynamics with riddled basins of attraction in
models of interacting populations, \emph{Chaos, Solitons and Fractals}
\textbf{12 }(2001), 301-311.

\bibitem {CA}L. Chen and K. Aihara, Strange attractors in chaotic neural
networks, \emph{IEEE Trans. Circuits Syst.} \textbf{47} (2000), 1455-1468.

\bibitem {CB}J. Chen, J. and D. Blackmore, On the exponentially
self-regulating population model, \emph{Chaos, Solitons \& Fractals}
\textbf{14} (2002), 1433-1450.

\bibitem {Dev}R. Devaney, \emph{An Introduction to Chaotic Dynamics},
Addision--Wesley, Boston, 1989.

\bibitem {Easton}R. Easton, Trellises formed by stable and unstable manifolds
in the plane, \emph{Trans. Amer. Math. Soc.} \textbf{294} (1986), 719-732.

\bibitem {Falc}K. Falconer, \emph{Techniques in Fractal Geometry}, Wiley,
Chichester, U.K., 1997.

\bibitem {JEAA}J. E. Franke and A.-A. Yakubu, Exclusion principle for
density-dependent discrete pioneer-climax models, \emph{J. Math. Anal. Appl.}
\textbf{187} (1994), 1019-1046.

\bibitem {Frank}J. E. Franke and A.-A. Yakubu, Pioneer exclusion in a one-hump
discrete pioneer-climax competitive system, \emph{J. Math. Biol.} \textbf{32}
(1994), 771-787.

\bibitem {GOPY}C. Grebogi, E. Ott, S. Pelikan and J. Yorke, Strange attractors
that are not chaotic, \emph{Physica D} \textbf{13 }(1984), 261-268.

\bibitem {Gh}J. Guckenheimer and P. Holmes, \emph{Nonlinear Oscillations,
Dynamical Systems and Bifurcations of Vector Fields}, Springer-Verlag, New
York, 1983.

\bibitem {Has}M. P. Hassell and M. N. Comins, Discrete time models for
two-species competition, \emph{Theoret. Population Biol.} \textbf{9} (1976), 202-221.

\bibitem {Hat}A. Hatcher, \emph{Algebraic Topology}, Cambridge Univ. Press,
Cambridge, 2002.

\bibitem {HKLN}B. Hunt, J. Kennedy, T-Y. Li and H. Nusse (eds), The Theory of
Chaotic Attractors, Springer-Verlag, New York, 2004.

\bibitem {JB1}Y. Joshi, Y. and D. Blackmore, Bifurcation and chaos in higher
dimensional pioneer-climax systems, \emph{Int'l. Electronic J. Pure and Appl.
Math.} \textbf{1} (3) (2010), 303-337.

\bibitem {JB2}Y. Joshi, Y. and D. Blackmore, Exponentially decaying discrete
dynamical systems, \emph{Recent Patents on Space Tech}. \textbf{2} (1) (2012), 37-48.

\bibitem {Mar}F. Marotto, Snap-back repellers imply chaos, \emph{J. Math.
Anal. Appl.} \textbf{63} (1978), 199-223.

\bibitem {Mil1}J. Milnor, On the concept of attractor, \emph{Comm. Math.
Phys.} \textbf{99} (1985), 177-195.

\bibitem {Mil2}J. Milnor, Correction and remarks: \textquotedblleft On the
concept of attractor\textquotedblright, \emph{Comm. Math. Phys.} \textbf{102}
(1985), 517-519.

\bibitem {Mis}M. Misiurewicz, Strange attractors for the Lozi mappings,
\emph{N.Y. Acad. Sci.} \textbf{357 }(1980), 348-358.

\bibitem {OS}W. Ott and M. Stenslund, From limit cycles to strange attactors,
\emph{Commun. Math. Phys.} \textbf{296} (2010), 215-249.

\bibitem {Rob}C. Robinson, \emph{Dynamical Systems: Stability, Symbolic
Dynamics, and Chaos}, CRC Press Inc., Boca Raton, 1995.

\bibitem {RBTUZ}A. Rosato, D. Blackmore, X. Tricoche, K. Urban and L. Zuo,
Dynamical systems models and discrete element simulations of a tapped granular
column, Powders \& Grains 2013, July 8-12, 2013, Sydney, Australia, AIP Conf.
Proc. 1541 (2013), pp. 317-320.

\bibitem {Rou}Z. Roupas, Phase space geometry and chaotic attractors in
dissipative Nambu mechanics, \emph{J. Phys. A: Math.Theor}. \textbf{45}
(2012), 195101.

\bibitem {Ru}D. Ruelle, \emph{Chaotic Evolution and Strange Attractors},
Lezioni Lincee, Cambridge Univ. Press, Cambridge, 1989.

\bibitem {Schaf}W. Schaffer, Order and chaos in ecological systems,
\emph{Ecology} \textbf{66} (1985), 93-106.

\bibitem {SK}W. Schaffer and M. Kot, Do strange attractors govern ecological
systems? \emph{BioScience} \textbf{35} (1985), 342-350.

\bibitem {JF}J. F. Selgrade, Planting and harvesting for pioneer-climax
models, \emph{Rocky Mountain J. Math.} \textbf{24} (1994), 293-310.

\bibitem {Seln}J. F. Selgrade and G. Namkoong, Stable periodic behavior in a
pioneer-climax model, \emph{Nat. Resour. Model} \textbf{4} (1990), 215-227.

\bibitem {SR1}J. F. Selgrade and J. H. Roberds, On the structure of attractors
for discrete, periodically forced systems with applications to population
models, \emph{Physica D} \textbf{158} (2001), 69-82.

\bibitem {SR2}J. F. Selgrade and J. H. Roberds, Global attractors for a
discrete selection model with periodic immigration, \emph{J. Diff. Eqs. and
Appl}. \textbf{15} (2007), 275-287.

\bibitem {Shub}M. Shub, \emph{Global Stability of Dynamical Systems},
Springer-Verlag, New York, 1987.

\bibitem {Sum}S. Sumner, Hopf bifurcation in pioneer-climax competing species
models, \emph{Math. Biosci.} \textbf{137} (1996), 1-24.

\bibitem {Thun}H. Thunberg, Periodicity versus chaos in one-dimensional
dynamics, \emph{SIAM Review} \textbf{43} (2001), 3-30.

\bibitem {UW}J. Ugarcovici and H. Weiss, Chaotic attractors and physical
measures for some density dependent Leslie population models,
\emph{Nonlinearity} \textbf{20} (2007), 2897-2906.

\bibitem {WY1}Q. Wang and L-S. Young, Strange attractors with one dimension of
instability, \emph{Commun. Math. Phys}. \textbf{218} (2001), 1-97.

\bibitem {WY2}Q. Wang and L-S. Young, From invariant curves to strange
attractors, \emph{Commun. Math. Phys}. \textbf{225} (2002), 275-304.

\bibitem {WY2a}Q. Wang and L-S. Young, Strange attractors in
periodically-kicked limit cycles and Hopf bifurcations, \emph{Commun. Math.
Phys}. \textbf{240} (2003), 509-529.

\bibitem {WY3}Q. Wang and L-S. Young, Toward a theory of rank one attractors,
\ \emph{Ann. of Math. (2)} \textbf{167} (2008), 349-480.

\bibitem {Wig}S. Wiggins, \emph{Introduction to Applied Nonlinear Dynamical
Systems and Chaos}, Springer, New York, 2003.

\bibitem {Zas}G. Zaslavsky, The simplest case of a strange attractor,
\emph{Phys. Lett. A} \textbf{69} (1978/79), 145-147.
\end{thebibliography}
\end{document}